\documentclass[a4paper,11pt,reqno]{amsart}
\usepackage{amsmath, amsthm, amssymb, mathrsfs, amsfonts}
\usepackage{array}
\usepackage[active]{srcltx}
\usepackage{graphicx}
\usepackage[margin=3cm]{geometry}
\usepackage{cite}
\usepackage{hyperref}
\usepackage{enumerate}
\newtheorem{teor}{Theorem}[section]

\newtheorem{lem}{Lemma}[section]
\newtheorem{prop}{Proposition}[section]

\newtheorem{obs}{Remark}

\def\car{\mathop{\rm char}\nolimits}

\title{Normal group algebras}

\author[A. Holguin-Villa]{Alexander Holgu\'in-Villa} 
\address{Alexander Holgu\'in-Villa, Escuela de Matem\'aticas, Universidad Industrial de Santander}
\email{aholguin@uis.edu.co; alexholguinvilla@gmail.com}

\author[J.H. Castillo]{John H. Castillo} 
\address{John H. Castillo, Departamento de Matem\'aticas y Estad\'istica, Universidad de Nari\~no}
\email{jhcastillo@udenar.edu.co}

\begin{document}
\maketitle
 \noindent
  \renewcommand{\refname}{References}


\begin{abstract}
  Let $\mathbb{F}G$ denote the group algebra of the group $G$ over the field $\mathbb{F}$ with
  $\car(\mathbb{F})\neq 2$. Given both a homomorphism $\sigma:G\rightarrow \{\pm1\}$ and a group
  involution $\ast: G\rightarrow G$, an oriented involution of $\mathbb{F}G$ is defined by
  $\alpha=\Sigma\alpha_{g}g \mapsto \alpha^\circledast=\Sigma\alpha_{g}\sigma(g)g^{\ast}$. In
  this paper, we determine the conditions under which the group algebra $\mathbb{F}G$ is normal,
  that is, conditions under which $\mathbb{F}G$ satisfies the $\circledast$-identity
  $\alpha\alpha^\circledast=\alpha^\circledast\alpha$. We prove that $\mathbb{F}G$ is normal if and only if the set of symmetric elements under $\circledast$ is commutative.
\end{abstract}

\section{Introduction}

Let $\mathbb{F}G$ denote the group algebra of the group $G$ over the field $\mathbb{F}$ with $\car(\mathbb{F})\neq 2$. Any involution
$\ast: G \rightarrow G$ can be extended $\mathbb{F}$-linearly to an algebra involution $\ast: \mathbb{F}G \rightarrow \mathbb{F}G$. A natural involution on $G$ is the so-called \emph{classical involution}, which maps $g \in G$ to $g^{-1}$. This involution appears as a technical tool to obtain results on units in a paper of G. Higman \cite{Hig:40}. In particular, it is used there to prove that if $G$ is a finite abelian group, then $\mathbb{Z}G$ has non-trivial units unless either the orders of the elements of $G$ divide four, or six, in which case $\mathbb{Z}G$ has only trivial units.

Let $\sigma:G\rightarrow \{\pm1\}$ be a group homomorphism (called an orientation). If $\ast:G \rightarrow G$ is a group involution, an
\emph{oriented group involution} $\circledast$ of $\mathbb{F}G$ is defined by

\begin{equation}\label{eq0}
\alpha=\sum_{g\in G} \alpha_gg \mapsto \alpha^\circledast=\sum_{g\in G} \alpha_g\sigma(g)g^{\ast}.
\end{equation}

Notice that, when $\sigma$ is non-trivial, $\car(\mathbb{F})$ must be different from $2$. It is clear that, $\alpha\mapsto \alpha^\circledast$
is an involution in $\mathbb{F}G$ if and only if $gg^\ast\in N=ker(\sigma)$ for all $g\in G$.

In the case that the involution on $G$ is the classical involution, $g\mapsto g^{-1}$, the map $\circledast$ is precisely the
oriented involution introduced by S.~P. Novikov (1970) in the context of $K$-theory, see \cite{Nov:70}.

Let $R$ be an $\mathbb{F}$-algebra. Recall that a subset $S$ of $R$ satisfies a polynomial identity ($S\in \text{PI}$ or $S$ is PI for short) if there
exists a nonzero polynomial $f(x_1,x_2,..,x_n)$ in the free associative $\mathbb{F}$-algebra $\mathbb{F}\{X\}$ on the countable set of non-commuting variables $X=\{x_1, x_2,...\}$ such that $f(s_1,s_2,...,s_n)=0$ for all $s_i\in S$. For instance, $R$ is commutative if it satisfies the polynomial
identity $f(x_1, x_2)=x_1x_2-x_2x_1$ and, any finite dimensional associative algebra satisfies the \emph{standard polynomial identity} of degree $n$, \cite[Lemma 5.1.6, p. 173]{Pas:77}

$$St_n(x_1,x_2,...,x_n)=\sum_{\rho\in \mathcal{S}_n}(sgn\rho)x_{\rho(1)}x_{\rho(2)}\cdot\cdot\cdot x_{\rho(n)}.$$

The conditions under which $\mathbb{F}G$ satisfies a polynomial identity were determined in classical results due to Isaacs and Passman summarized as follows.

\begin{teor}\label{teor1}\emph{(\cite[Corollaries 3.8 and 3.10, p. 196-197]{Pas:77})}
Let $\mathbb{F}$ be a field of characteristic $p$ and $G$ a group. Then $\mathbb{F}G$ satisfies a PI
if and only if $G$ has a $p$-abelian subgroup of finite index.
\end{teor}

Recall that, for a prime $p$, a group $G$ is called $p$-abelian if $G'$, the commutator subgroup of $G$,
is a finite $p$-group and $0$-abelian means abelian.

\vspace{0.3cm}

Let $R$ be a ring with $1_R=1$ and $\star$ an involution on $R$. Let us denote by $R^+=\left\{r\in R: r^\star=r\right\}$ and $R^-=\left\{r\in R: r^\star=-r\right\}$ the sets of \emph{symmetric and skew-symmetric elements} respectively of the ring $R$ under the involution $\star$. We denote with $\mathcal{U}(R)$ the unit group of $R$ and $\mathcal{U}^+(R)=\mathcal{U}(R) \cap R^+$ for the set of symmetric units. A question
of general interest is which properties of $R^+$ or $R^-$ can be lifted to $R$. A similar question may be posed for the set of the symmetric units
or the subgroup that they generate, i.e., to determine the extent to which the properties of $\mathcal{U}^+(R)$ determine either the properties of
the whole unit group $\mathcal{U}(R)$ or the whole ring $R$. After the fundamental work of Amitsur \cite{Am:68, Am:69}, and the interest in rings
with involution developed from the 1970s by Herstein and his collaborators, \cite{Her}, it is natural to consider group algebras from this viewpoint.

If $R$ is a $\mathbb{F}$-algebra with involution $\star$ such that $\lambda^\star=\lambda$, for all $\lambda\in \mathbb{F}$, we may define an involution
on the free associative algebra $\mathbb{F}\{X\}$ (again $X$ is a countable set of non-commuting variables) by setting $x_{2n+1}^\star=x_{2n+2}$,
for all $n\geq 0$. Then, after renumbering we get the free associative algebra with involution, $\mathbb{F}\{X, \star\}=\mathbb{F}\{x_1, x_1^\star, x_2, x_2^\star,...\}$. Of course, an element in this algebra is a polynomial in the variables $x_i$ and $x_i^\star$, which do not commute. We say that
$0\neq f(x_1, x_1^\star,..., x_n, x_n^\star)\in \mathbb{F}\{x_1, x_1^\star, x_2, x_2^\star,...\}$ is a $\star$-polynomial identity
($\star$-PI for short) for a subset $S$ of $R$, if $f(s_1, s_1^\star,..., s_n, s_n^\star)=0$ for all $s_1, s_2,..., s_n\in S$. With this terminology in mind, the results of Amitsur can be written in the following way.

\begin{teor}\label{teor2}\emph{(\cite[Theorem 1, p. 63]{Am:69})}
Let $R$ be an $\mathbb{F}$-algebra with involution $\star$. If $R$ satisfies a $\star$-PI
$f(x_1, x_1^\star,..., x_r, x_r^\star)$ of degree $d$, then $R$ satisfies $St_{2d}(x_1,x_2,...,x_{2d})^m$ for some $m$, the \emph{standard polynomial identity} in $2d$ variables. If $R$ is \emph{semi-prime} then $m=1$, i.e., if $R$ is $\star$-PI,
then $R$ is (in usual sense) PI.
\end{teor}

It is obvious that if $R^+$ (or $R^-$) satisfies a polynomial identity $f(x_1, x_2,..., x_n)$, then $R$ satisfies the $\star$-PI $f(x_1+x_1^\star,...,x_n+x_n^\star)$ (respectively $f(x_1-x_1^\star,..., x_n-x_n^\star)$) and, by the above theorem, it is PI.

Notice that Amitsur's result proves the existence of an ordinary polynomial identity for the $\mathbb{F}$-algebra $R$ however, in general, does not give
any information on its degree. The reason for this failure is the following: the theorem was proved first for semi-prime rings where, through structure
theory, the degree of an identity for $R$ is well related to that of the given $\star$-identity; then the result was pushed to arbitrary rings by means
of the so-called Amitsur's trick. In that procedure any information on the degree of the $\star$-identity satisfied by $R$ is lost. This problem was
solved by Bahturin, Giambruno and Zaicev in \cite{BaGiZa}. In fact, by using combinatorial methods pertaining to the asymptotic behaviour of a numerical sequence attached to the algebra $R$, it was shown that one can relate the degree of a $\star$-polynomial identity satisfied by $R$ to the degree of a polynomial identity for $R$ by mean of an explicit function.

Another question in this direction is the following: if $R$ satisfies some special kind of $\star$-polynomial
identity, what kind of ordinary identity can one get in Amitsur's result? Recalling that $R^-$ is a Lie
subalgebra of $R$ under the bracket operation $[a, b]=ab-ba$, it is natural to ask if, in particular, the Lie
nilpotence of $R^-$ implies the Lie nilpotence (or some other special type of identity) of $R$. The best known
result in this direction is due to Zalesskii and Smirnov.

\begin{teor}[\cite{ZS:81}]
Suppose that $R=\langle R^-, 1\rangle$ as a ring and that $\car(R)\neq 2$. If $R^-$ is Lie nilpotent then
$R$ is Lie nilpotent.
\end{teor}

An $\mathbb{F}$-algebra $R$ with involution $\star$ is said to be \emph{normal} if $rr^\star=r^\star r$, for all $r\in R$, \cite{Her}. In this
paper, we characterize group algebras $\mathbb{F}G$ which are normal in regard to the oriented group involution $\circledast$. This question
was studied by Bovdi and Siciliano in \cite{Bov:08}. Our methods provide shorter and natural proofs of some of the results in \cite{Bov:08}, using the concept of LC-group and as an application of them we show that $\mathbb{F}G$ is normal if and only if $\mathbb{F}G^+$ is commutative, see \cite{OP}.

Throughout this paper $\mathbb{F}$ will always denote an infinite field with $\car(\mathbb{F})\neq 2$, $G$ a group, $\ast$ and $\sigma$ an
involution and an orientation of $G$, respectively. We will denote with $\circledast$ an oriented group involution of $\mathbb{F}G$ given
by the expression \eqref{eq0}.

\section{Prerequisites and assumed results}

We gather important facts to use in later sections. In order to state our results, we need a definition. Let $\zeta=\zeta(G)$
denote the center of $G$. We recall the reader that a group $G$
is said to be an LC-group (that is, it has the ``limited commutativity'' property)  if it is non-abelian and for any pair of elements $g, h\in G$, we have that $gh=hg$ if and only if at least one element of $\{g, h, gh\}$ lies in $\zeta$. This family of groups was introduced by E. Goodaire. Moreover, by \cite[Proposition III.3.6, p. 98]{GJP:96} LC groups with a unique non-identity commutator $s$ (obviously it has order $2$ and is central) are precisely those non-abelian groups with $G/\zeta(G)\cong C_2\times C_2$, where $C_2$ is the cyclic group of order $2$. If $G$ has an involution $\ast$, then we say that $G$ is an special LC-group, or SLC-group, if it is an LC-group, it has an unique non-identity commutator $s$
and the involution  $\ast$ is given by

\begin{equation}\label{eq1}
g^* = \begin{cases}
         g, & \text{if $g$ is central;} \\
         sg, & \text{otherwise;}
      \end{cases}
\end{equation}
and we refer to this as the \emph{canonical} involution on an SLC-group.

The multiplicative commutator $g^{-1}h^{-1}gh$ of $g, h\in G$ will be denoted by $(g,h)$ and $g^h=h^{-1}gh$.

\begin{obs}
If $G$ is a group with the LC-property, then for all $g\in G$, $g^2$ is central. Thus, since  $(g,h)=g^{-1}h^{-1}gh=g^{-2}gh^{-1}gh^{-1}h^2=g^{-2}(gh^{-1})^2h^2$, commutators are central in
a LC-group $G$.
\end{obs}

We recall that a non-abelian group $G$ is a Hamiltonian group if every subgroup of $G$ is normal. It is well-known that in this
case $G\cong Q_8\times E\times O$, \cite[Theorem 1.8.5, p. 63]{PS:02}, where $Q_8=\langle x, y: x^4=1, x^2=y^2, y^{-1}xy=x^{-1}\rangle$ is
the quaternion group of order $8$, $E$ is an elementary abelian $2$-group and $O$ is an abelian group with every element of odd
order. When $O=\{1\}$, $G$ is called a Hamiltonian $2$-group.

In \cite{JRM:06} Jespers and Ruiz Mar\'in characterized when for an arbitrary commutative ring $R$ and a group $G$ the set $RG^+$ forms
a commutative ring with respect to the induced involution $\ast$. In particular, when $\car(R)\neq 2$ the result is as follows:

\begin{lem}\label{lem1}\emph{(\cite[Theorem 2.4, p.730]{JRM:06})}
Let $R$ be a commutative ring with $\car(R)\neq 2$, $G$ a non-abelian group with an involution $\ast$ which is extended $R$-linearly to
$RG$. Then, $RG^+$ is commutative if and only if $G$ is an SLC-group. For the classical involution on $G$ ($g\mapsto g^{-1}$), it follows
that $RG^+$ is commutative if and only if $G$ is a Hamiltonian $2$-group, \cite[Theorem 2.2, p. 3351]{Cristo:06}.
\end{lem}

In an associative ring $R$, we define the Lie product via $[x_1, x_2]=x_1x_2-x_2x_1$ and, we can extended this recursively via

$$[x_1,...,x_n,x_{n+1}]=[[x_1,...,x_n], x_{n+1}].$$

Let $S$ be a subset of R. We say that $S$ is \emph{Lie nilpotent} if there exists an $n\geq 2$ such that $[a_1,...,a_n]=0$ for all $a_i\in S$. The smallest such $n$ is called the nilpotency index of $S$. For a positive integer $n$, we say that $S$ is Lie $n$-Engel if

$$[a,\underbrace{b,...,b}_{n \text{ times}}]=0$$
for all $a, b\in S$. Obviously if $S$ is Lie nilpotent then it is Lie $n$-Engel for some $n$.

In the study of the Lie nilpotence and the Lie n-Engel properties in the sets $\mathbb{F}G^+$ and $\mathbb{F}G^-$, Giambruno and Sehgal
\cite[Theorem 1, p. 4255]{GiSe} and Giambruno, Polcino Milies and Sehgal \cite[Lemma 2.4, p. 891]{GPS:09} showed the following general result
about semiprime rings with involution.

\begin{lem}\label{lem2}
Let $R$ be a semiprime ring with involution $\star$ such that $2R=R$. If $R^-$ (respectively $R^+$) is Lie nilpotent (Lie $n$-Engel), then $[R^-,R^-]=0$ (respectively $[R^+,R^+]=0$) and $R$ satisfies $St_4(x_1, x_2, x_3, x_4)$, the standard polynomial identity in four non-commuting variables, i.e., $R$
satisfies
$$
St_4(x_1,x_2,x_3,x_4)=\sum_{\rho\in \mathcal{S}_4}(sgn\rho)x_{\rho(1)}x_{\rho(2)}x_{\rho(3)}x_{\rho(4)}.
$$
\end{lem}

\section{Group involution}

Let $\mathbb{F}$ be a field with $\car(\mathbb{F})\neq 2$ and let $G$ be a group. Suppose that $\ast: G\rightarrow G$ is a function satisfying $(gh)^\ast=h^\ast g^\ast$ and $(g^\ast)^\ast=g$ for all $g, h\in G$. Extending it $\mathbb{F}$-linearly, we obtain an involution on $\mathbb{F}G$, the so-called induced involution (obviously, the classical involution is the one induced from $g\mapsto g^{-1}$ on $G$). We denote $G^+=\{g\in G: g^\ast=g\}$ the set of symmetric elements of $G$ under $\ast$. Now, suppose that the group algebra $\mathbb{F}G$ is normal with respect to the group involution $\ast$, i.e., $\mathbb{F}G$ satisfies the $\ast$-PI $\alpha\alpha^\ast=\alpha^\ast\alpha$.

We begin this section with some lemmas, results where we establish necessary conditions under which the group algebra $\mathbb{F}G$
is normal.

\begin{lem}\label{lem3}
Suppose that $\mathbb{F}G$ is normal and let $g, h \in G$. Then either $gh=hg$ or $gh=g^\ast h^\ast$.
\end{lem}
\begin{proof}
As $\mathbb{F}G$ is normal, we have that $(g+h^\ast)(g^\ast+h)=(g^\ast+h)(g+h^\ast)$. Since $g^\ast g=gg^\ast$ and
$hh^\ast=h^\ast h$, we get
$$
gh+h^\ast g^\ast=hg+g^\ast h^\ast,
$$
and since these are just group elements, $gh=hg$ or $gh=g^\ast h^\ast$, as required.
\end{proof}

\begin{lem}\label{lem4}
Suppose that $\mathbb{F}G$ is normal. Then for all $g\in G$, $g^2\in \zeta(G)$.
\end{lem}
\begin{proof}
Assume that $g\in G$ is not central and let $h\in G$ such that $gh\neq hg$. Lemma $\ref{lem3}$, applied to the non-commuting
elements $g$ and $gh$, yields that $g^2h=g^\ast(gh)^\ast=(g^\ast h^\ast)g^\ast$. Then,
$$
g^2h=g^\ast(h^\ast g^\ast)=(hg)^\ast g^\ast.
$$
Since $g^\ast, (hg)^\ast\in G$ do not commute, then again by Lemma $\ref{lem3}$, $(hg)^\ast g^\ast=(hg)g$. Therefore,
$g^2h=hg^2$ and hence $g^2\in \zeta(G)$.
\end{proof}

\begin{lem}\label{lem5}
Suppose that $\mathbb{F}G$ is normal. Then $G^+\subseteq \zeta(G)$. In particular, $gg^\ast=g^\ast g \in \zeta(G)$, for all
$g\in G$.
\end{lem}
\begin{proof}
Let $g\in G^+$ and suppose that there exists $h\in G$ such that $gh\neq hg$. By Lemma $\ref{lem3}$,
$$
gh=g^\ast h^\ast=gh^\ast,
$$
forcing $h\in G^+$. Then, again as $\mathbb{F}G$ is normal,

\begin{align*}
(gh+g)(hg+g)   &= (hg+g)(gh+g) \nonumber \\
gh^2g+2ghg+g^2 &= hg^2h+hg^2+g^2h+g^2.
\end{align*}

By Lemma $\ref{lem4}$, $g^2, h^2\in \zeta(G)$, hence $2ghg=2g^2h$ and since the characteristic is not $2$, $gh=hg$, a contradiction.
Hence $g\in \zeta(G)$, as desired.
\end{proof}

In the following result we establish necessary and sufficient conditions on $G$ under which the group algebra $\mathbb{F}G$ is
normal with respect to group involution $\ast$. Clearly, if $G$ is an abelian group, then $\mathbb{F}G$ is commutative and thus
for all $\alpha\in \mathbb{F}G$, $\alpha\alpha^\ast=\alpha^\ast\alpha$.

\begin{teor}\label{teor3}
Let $\mathbb{F}$ be a field with $\car(\mathbb{F})\neq 2$ and let $G$ be a non-abelian group with an involution $\ast$ extended
$\mathbb{F}$-linearly to $\mathbb{F}G$. The following conditions are equivalent:
\begin{enumerate}
 \item[$1.$] $\mathbb{F}G$ is normal;
 \item[$2.$] $G$ is an SLC-group;
 \item[$3.$] $\mathbb{F}G^+$ is commutative.

\end{enumerate}
\end{teor}

\begin{proof}
In view of Lemma \ref{lem1}, it remains only to prove that $(1)$ is equivalent to $(2)$.

Assume that $\mathbb{F}G$ is normal. Let $g\in G\setminus \zeta(G)$ and let $h\in G$ such that $gh\neq hg$. Then $gh$ does not
commute with $h^{-1}$ and so, by Lemma $\ref{lem3}$,
$$
h^{-1}(gh)=(h^{-1})^\ast(gh)^\ast=(h^\ast)^{-1}(h^\ast g^\ast)=g^\ast.
$$
Thus, $g^{-1}h^{-1}gh=g^{-1}g^\ast$ and so $s=g^{-1}g^\ast \in G'$

Since, by Lemma \ref{lem4} $g^2\in \zeta(G)$,  we have that $G/\zeta(G)$ is an elementary abelian $2$-group. Then
$G'\subseteq \zeta(G)$ and therefore $s$ is central. Moreover,
\begin{align*}
s^2  &= (g^{-1}h^{-1}gh)(g^{-1}h^{-1}gh) \nonumber \\
     &= g^{-1}(g^{-1}h^{-1}gh)(h^{-1}gh) \nonumber \\
     &= g^{-2}(h^{-1}g^2)h \nonumber \\
     &= g^{-2}(g^2h^{-1})h \tag{$g^2\in \zeta(G)$} \nonumber \\
     &= 1.
\end{align*}
Since $s\in \zeta(G)$, $s^2=1$ and $s=g^{-1}g^\ast$ we get that $(g^\ast)^2=g^2$ and so, $g^{-1}g^\ast=g(g^\ast)^{-1}$. Because
of Lemma $\ref{lem3}$, $hg=h^\ast g^\ast$, and thus we obtain
$$
h^{-1}h^\ast=g(g^\ast)^{-1}=g^{-1}g^\ast=s,
$$
for all $g, h\in G$ such that $gh\neq hg$ and thus $G'=\{1, s\}$, i.e., $s$ is the unique non-identity commutator.

Notice that if $g\notin \zeta(G)$, from $s=g^{-1}g^\ast$ we get $g^\ast=sg$. Moreover, if $g\in \zeta(G)$ and $h\notin \zeta(G)$,
then $gh\notin \zeta(G)$. Thus,
$$
shg\underbrace{=}_{g\in \zeta(G)}sgh=(gh)^*=h^*g^*=shg^*.
$$
Hence $g^\ast=g$ and thus the involution $\ast$ is trivial on $\zeta(G)$ and only on $\zeta(G)$.

Next we show that $G$ has the LC-property. Let $g, h\in G$ be such that $gh=hg$ but $g, h\notin \zeta(G)$. Then,
$$
(gh)^\ast=h^\ast g^\ast=(sh)(sg)=s^2hg=gh.
$$
Hence $gh\in G^+$ and thus, by Lemma $\ref{lem5}$, $gh\in \zeta(G)$, as desired. Thus $(1)$ implies $(2)$.

We now prove that $(2)$ implies $(1)$. Assume that $G$ has the LC-property and a unique non-identity commutator $s$. Since $\ast$
is given by expression \eqref{eq1}, if $g\notin \zeta(G)$ then $g^\ast=sg$ and so, for all $\alpha \in \mathbb{F}G$,
$$
\alpha=\sum_{g\in G}\alpha_gg=\sum_{z\in\zeta(G)}\alpha_zz+\sum_{g\notin \zeta(G)}\alpha_gg=\beta+\gamma.
$$
It follows that $\alpha^\ast=\beta+s\gamma$. Thus,
\begin{align*}
\alpha\alpha^* &=(\beta+\gamma)(\beta+s\gamma) \nonumber \\
               &=\beta^2+s\beta\gamma+\gamma\beta+s\gamma^2 \nonumber \\
               &=\beta^2+s\gamma\beta+\beta\gamma+s\gamma^2 \tag{$\beta\in \zeta(\mathbb{F}G)$} \nonumber \\
               &=(\beta+s\gamma)\beta+(\beta+s\gamma)\gamma \nonumber \\
               &=(\beta+s\gamma)(\beta+\gamma) \nonumber \\
               &=\alpha^*\alpha.
\end{align*}
Hence, $\mathbb{F}G$ is normal.
\end{proof}

\section{Oriented group involutions}

Let $\mathbb{F}$ be a field and let $G$ be a group with a non-identity homomorphism $\sigma:G\rightarrow \{\pm1\}$ and an
involution $\ast:G\rightarrow G$. Since $\sigma$ is non-identity, $\car(\mathbb{F})$ must be different from $2$. Let $\circledast:\mathbb{F}G\rightarrow \mathbb{F}G$ denote the involution obtained as a linear extension of the involution $\ast$,
twisted by the homomorphism $\sigma:G\rightarrow \{\pm1\}$.

If $N$ denotes the kernel of $\sigma$, then $N$ is a subgroup in $G$ of index $2$. It is clear that the involution $\circledast$
coincides on the subalgebra $\mathbb{F}N$ with the group involution $\ast$. Also, we have that the symmetric elements in $G$,
under $\circledast$, are the symmetric elements in $N$ under $\ast$. If we denote the set of symmetric elements in $G$, under
$\circledast$, by $N^+$, then we can write $N^+=N\cap G^+$, where as mentioned earlier, $G^+$ denotes the set of symmetric elements
of G in regard to $\ast$. Thus, if $\mathbb{F}G$ satisfies a $\circledast$-PI, then $\mathbb{F}N$ satisfies a $\ast$-PI.

Now, suppose that $\mathbb{F}G$ is normal. Then $\mathbb{F}N$ is a normal group algebra. Hence, by Theorem \ref{teor3}, the structure
of $N$ and the action of $\ast$ on $N$ are known, more exactly, $N$ is abelian or SLC-group. With the aim to determine the structure
of $G$ and the action of $\circledast$ we shall begin with some lemmas.

\begin{lem}\label{lem6}
Let $g, h \in G$, and assume $\mathbb{F}G$ is normal. Then one of the following holds:
\begin{enumerate}
 \item[$1$.] If $\sigma(g)\sigma(h)=1$, then either $gh=hg$ or $gh=g^\ast h^\ast$,
 \item[$2.$] If $\sigma(g)\sigma(h)=-1$, then either $gh=hg$ or $gh=(gh)^\ast$.

\end{enumerate}
\end{lem}

\begin{proof}
By hypothesis $(g+h^\circledast)(g^\circledast+h)=(g^\circledast+h)(g+h^\circledast)$. Since $xx^\circledast=x^\circledast x$ for
all $x\in G$, we obtain that
$$
gh+h^\circledast g^\circledast=g^\circledast h^\circledast+hg.
$$
\begin{enumerate}
 \item[$1.$] If $\sigma(g)\sigma(h)=1$, then $gh+h^\ast g^\ast=g^\ast h^\ast+hg$ and, since $\car(\mathbb{F})\neq 2$ and as each
      side is a sum of group elements, the result follows.
 \item[$2.$] If $\sigma(g)\sigma(h)=-1$, then $gh+g^\ast h^\ast=hg+h^\ast g^\ast=hg+(gh)^\ast$, so either $gh=hg$ or $gh=(gh)^\ast$.
      This finishes the proof.
\end{enumerate}
\end{proof}

\begin{lem}\label{lem7}
Let $g,h\in G$ such that $gh\neq hg$ and assume $\mathbb{F}G$ is normal. Then one of the following holds:
  \begin{enumerate}
   \item[$1.$] If either $\sigma(g)=\sigma(h)=1$ or if $\sigma(g)=-1$ and $\sigma(h)=1$, then $g^2h=hg^2$,
   \item[$2.$] If either $\sigma(g)=1$ and $\sigma(h)=-1$ or if $\sigma(g)=\sigma(h)=-1$, then $g^2h=(g^2h)^*$.
  \end{enumerate}
In particular, for $m, n \in N$, $(m^2, n)=1$.
\end{lem}

\begin{proof}
 We have the following possibilities:
\begin{enumerate}
 \item[$1.$] If either $\sigma(g)=\sigma(h)=1$ or if $\sigma(g)=-1$ and $\sigma(h)=1$, Lemma $\ref{lem6}(1)$, applied to the
      non-commuting elements $g$ and $gh$, yields that $g^2h=g(gh)=g^\ast(h^\ast g^\ast)=(hg)^\ast g^\ast$. Again by Lemma
      $\ref{lem6}(1)$ (applied to the non-commuting elements $(hg)^\ast$ and $g^\ast$), we obtain that
      $$
      g^2h=(hg)^*g^*=(hg)g=hg^2.
      $$
 \item[$2.$] If either $\sigma(g)=1$ and $\sigma(h)=-1$ or if $\sigma(g)=\sigma(h)=-1$, then by Lemma $\ref{lem6}(2)$, applied
      again to the non-commuting elements $g$ and $gh$, yields that
      $$
      g^2h=g(gh)=[g(gh)]^\ast=(g^2h)^\ast.
      $$
      In particular, as $\mathbb{F}N$ is normal it follows from Lemma $\ref{lem4}$ that $m^2\in \zeta(N)$.
\end{enumerate}
\end{proof}

\begin{lem}\label{lem8}
If $\mathbb{F}G$ is normal, then $N^+\subseteq \zeta(G)$. 
\end{lem}

\begin{proof}
Let $n\in N^+$ and $g\in G$. We need to prove that $ng=gn$. Since $\mathbb{F}N$ is normal, by Lemma $\ref{lem5}$, it is clear
if $g\in N$. Assume now that $g\notin N$. Then by part $(2)$ of Lemma $\ref{lem6}$, $ng=g^\ast n^\ast=g^\ast n$ and thus
$g^\ast=ngn^{-1}$. Since $\mathbb{F}G$ is normal,
\begin{align*}
      (ng+g)((ng)^\circledast+g^\circledast)                       &=((ng)^\circledast +g^\circledast)(ng+g) \nonumber \\
			(ng+g)(\sigma(g)g^\ast n+\sigma(g)g^\ast)                       &=(\sigma(g)g^\ast n+\sigma(g)g^\ast)(ng+g) \nonumber \\
      (ng+g)(g^\ast n+g^\ast)                       &=(g^\ast n+g^\ast)(ng+g) \nonumber \\
      (ng)(g^\ast n)+(ng)g^\ast+g(g^\ast n)+gg^\ast &=(g^\ast n)(ng)+(g^\ast n)g+g^\ast(ng)+g^\ast g \nonumber \\
      (ng)^2+(ng)^2n^{-1}+gng                       &=(ng)^2+2ng^2. 
\end{align*}
Furthermore by Lemma $\ref{lem7}(1)$, $(ng)^2n^{-1}=n^{-1}(ng)^2=gng$. Hence $2gng=2ng^2$ and as $\car(\mathbb{F})\neq 2$, it
follows that $gn=ng$. 

%
\end{proof}

We shall study separately the two possibilities for $N$.

\begin{lem}\label{lem9}
Let $\mathbb{F}$ be a field and let $G$ be a non-abelian group such that $\mathbb{F}G$ is normal. Let $\sigma:G\rightarrow \{\pm 1\}$
be a non-trivial orientation and $\ast$ an involution on $G$. Suppose that $N=ker(\sigma)$ is abelian. Then $x^\ast=x$, for all $x\in G\setminus N$ and $n^\ast=a^{-1}na=ana^{-1}$, for all $n\in N$ and for all $a\in G\setminus N$.
\end{lem}

\begin{proof}
Let $x$ be an arbitrary element in $G\setminus N$. Since the index of $N$ in $G$ is equal to $2$, it follows that $G=N\cup Nx$. Thus,
as $N$ is abelian, clearly, if $x$ is central, then $G$ is abelian, a contradiction. Therefore, there exists $n\in N$ such that
$nx\neq xn$. We have two possibilities:
\begin{enumerate}
 \item[$1.$] Assume that $nx\in G^+$. By Lemma $\ref{lem6}(1)$, we have that $(nx)x=(nx)x^\ast$ and thus $x^\ast=x$.
 \item[$2.$] Assume now that $nx\notin G^+$. Again by Lemma $\ref{lem6}(1)$, $(nx)x=(nx)^\ast x^\ast=x^\ast(n^\ast x^\ast)$. Then by
      Lemma $\ref{lem6}(2)$, applied to the non-commuting elements $x$ and $n$, yields that $xn=(xn)^\ast=n^\ast x^\ast$.
      Hence $(nx)x=x^\ast(n^\ast x^\ast)=x^\ast(xn)=xx^\ast n=nxx^\ast$, because $xx^\ast=x^\ast x\in \zeta(G)$. Again, we obtain that $x^\ast=x$.
\end{enumerate}
Let $n\in N$ and $a\notin N$. Then, $na, an\notin N$ and by the previous argument, we have that $na=(na)^\ast=a^\ast n^\ast=an^\ast$ and
$an=(an)^\ast=n^\ast a$, i.e.,
$$
n^\ast=a^{-1}na=ana^{-1},
$$
and the result follows.
\end{proof}

\begin{lem}\label{lem10}
Let $\mathbb{F}$ be a field and let $G$ be a non-abelian group such that $\mathbb{F}G$ is normal. Let $\sigma:G\rightarrow \{\pm 1\}$
be a non-trivial orientation and $\ast$ an involution on $G$. Suppose that $N=ker(\sigma)$ is an SLC-group. Then $G$ has a unique non-identity commutator $s$, which is central in $G$.
\end{lem}

\begin{proof}
Since $N$ is an SLC-group, the restriction $\ast_{|_N}:N \rightarrow N$ is the canonical involution given by the expression \eqref{eq1}.

Let $m, n\in N$ be such that $mn\neq nm$. Then $m^\ast=n^{-1}mn, n^\ast=m^{-1}nm$ and $s=m^{-1}m^\ast$.
We notice that this is valid for all non-central element $m\in N$. As $N'=\{1, s\}$ is characteristic in $N$ and $N\unlhd G$,
we have that $N'\unlhd G$. Hence, $g^{-1}sg=s$ for all $g\in G$ and so, $s$ is central.

To prove that $G'=\{1,s\}$, we shall consider the following cases.

\begin{enumerate}
 \item[$1.$] Let $g\notin N$ and assume that $mg\neq gm$ for some $m\in N$. Lemma $\ref{lem6}(2)$, applied to the non-commuting
      elements $m$ and $mg$, yields that $m^2g=m(mg)=[m(mg)]^\ast=(mg)^\ast m^\ast$. As $mg\neq gm$, then again by Lemma $\ref{lem6}(2)$, we have that $mg=(mg)^\ast$ and thus
      \begin{equation}\label{eq2}
       m^2g=m(gm^\ast),
       \end{equation}
       that is, $m^*=g^{-1}mg$.
 
Then, as $s=m^{-1}m^\ast$, we get that $s=(m,g)$. Moreover, we can prove that $$(ng,m)=(n,m)^g(g,m)=1 \text{ or } s,$$
for any $n\in N$.

 \item[$2.$] Let $g\notin N$. Suppose that $gm=gm$ for all $m\in N$. Then, for all $n\in N$, we have that
       \begin{align*}
        (ng,m) &=(n,m)^g(g,m) \nonumber \\
               &=(n,m)=1 \text{ or } s.
       \end{align*}

 \item[$3.$] Finally, assume that $g, h$ are any two elements of $G\setminus N$. Then, since $G=N\cup Ng$, $h=ng$, for some $n\in
      N$. Thus
      $$
      (h,g)=(ng,g)=(n,g)^g(g,g)=g^{-1}n^{-1}g^{-1}ng^2.
      $$
      Because of Lemma $\ref{lem7}(1)$, $ng^2=g^2n$, and thus we obtain $(h, g)=(g, n)$. Hence, by the first case $(h, g)=1$ or $s$. 
\end{enumerate}
Thus, in any case, $G'=\{1, s\}=N'\subseteq \zeta(G)$.
\end{proof}

\begin{lem}\label{lem11}
In the same conditions of the above lemma, there exits $g_0\in G\setminus N$ such that $g_0\in C_G(N)$ and $g_0^\ast=sg_0$.
Consequently, $g_0$ is central in $G$.
\end{lem}

\begin{proof}
Since $N$ is an SLC-group and as we mentioned earlier, \cite[Proposition III.3.6, p. 98]{GJP:96}, we have that $N/\zeta(N)\cong
C_2\times C_2$, i.e., $N=\langle x, y, \zeta(N)\rangle$, where $s=(x, y)$ is the unique non-identity commutator.

First we prove that there exits $g_0\in (G\setminus N)\cap C_G(N)$. Let $g\in G\setminus N$.

\vspace{0.3cm}

\emph{\textbf{Claim:}} If $g\notin C_G(N)$, then $gz=zg$ for all $z\in \zeta(N)$. 

In fact, by the proof of the case (1) of Lemma \ref{lem10}, we know that $y^*=y$ for all $y\in (G\setminus N)\setminus C_G(N)$. So $g^\ast=g$ and for any $z\in \zeta(N)$ we obtain that $$gz=(gz)^\ast=z^\ast g^\ast=zg.$$


%

The proof continues and we have three possible cases.
\begin{enumerate}
 \item[$1.$] Suppose that $(x,g)=(y,g)=1$. We shall prove that $g\in C_G(N)$. Assume that $g\not \in C_G(N)$, by the previous claim we have that $g$ commutes with any element of $\zeta(N)$. As $N=\langle x, y, \zeta(N)\rangle$, it implies that $g\in C_G(N)$ a contradiction. So in this case it is enough to take $g_0=g$.

 \item[$2.$] Suppose now that $(x,g)=s$ and $(y,g)=1$.
      We know that $(x, y)=s$. Let $g_0=gy$. It is clear that $(g_0, y)=1$. Furthermore,
      $$
      g_0x=(gy)x=sg(xy)=s^2x(gy)=xg_0,
      $$
      and thus $(g_0,x)=1$.

 \item[$3.$] In case $(x, g)=(y, g)=s$, then set $y'=xy$. Therefore,
      $$
      gy'=g(xy)=sx(gy)=s^2(xy)g=y'g.
      $$
      Hence $(y', g)=1$ and $(x, g)=s$. By the second case, it is enough to take $g_0=gy'$.
      $g_0\in C_G(N)$.

\end{enumerate}
Thus, in all cases, there exists $g_0\in (G\setminus N)\cap C_G(N)$.

Finally, let $m, n\in N$ such that $mn\neq nm$. As $g_0\in C_G(N)$, then $(g_0n)m\neq m(g_0n)$. Also, as $g_0n\notin C_G(N)$, by the proof of the claim we get that, $(g_0n)^\ast=g_0n=ng_0$, so
$$
ng_0=(g_0n)^\ast=n^\ast g_0^\ast=sng_0^\ast.
$$
Hence $g_0^\ast=sg_0$ and thus $g_0$ is as given in the statement.

Since $G=N\cup Ng_0$ and $g_0\in C_G(N)$, the last statement is now clear.

\end{proof}

\begin{prop}\label{prop1}
Let $\mathbb{F}$ be a field and let $G$ be a non-abelian group such that $\mathbb{F}G$ is normal. Let $\sigma:G\rightarrow \{\pm 1\}$
be a non-trivial orientation and $\ast$ an involution on $G$. Suppose that $N=ker(\sigma)$ is an SLC-group. Then $G$ is an LC-group and $\ast$ is given by
\begin{equation*}
g^{\ast}=
\begin{cases}
 g, & \text{ if $g\in N\cap \zeta(G)$ or $g\in (G\setminus N)\setminus \zeta(G)$;}\\
 sg, & \text{otherwise.}
\end{cases}
\end{equation*}

\end{prop}

\begin{proof}
By Lemma $\ref{lem10}$, $N'=\{1, s\}=G'\subset\zeta(G)$, i.e., $G$ has a unique non-trivial commutator $s$, which is a central element
of order $2$. Moreover, by Lemma $\ref{lem11}$, there exists $g_0\in (G\setminus N)\cap\zeta(G)$ such that $g_0^\ast=sg_0$. Consequently, $G=N\cup Ng_0=N\cup g_0N$.

Let $g, h\in G$ be such that $gh=hg$. Since $N$ is an LC-group the statement is clear if $g, h\in N$.

We have two possible cases.

\begin{enumerate}
 \item[$1.$] If $g, h\notin N$, then $g=g_0m$ and $h=g_0n$ for some $m, n\in N$, so
      $$
      (g_0m)(g_0n)=(g_0n)(g_0m).
      $$
      Now, as $g_0$ central in $G$, it follows that $mn=nm$. Since $N$ has the LC-property, then at least one element of $\{m, n, mn\}$ is central in $N$. Hence, either $g_0m, g_0n$ or $(g_0m)(g_0n)\in \zeta(G)$, as desired.

 \item[$2.$] If $g\in N$ and $h\notin N$, then $h=g_0m$ for some $m\in N$ and thus $g(g_0m)=(g_0m)g$. It follows that $gm=mg$. Again
      since $N$ is an LC-group, we have that either $g\in \zeta(N)$ or $m\in\zeta(N)$ or $gm\in\zeta(N)$. Consequently, at least one of the elements of the set $\{g, g_0m, g(g_0m)\}$ is central in $G$.


\end{enumerate}
Thus, in any case, $G$ is an LC-group with a unique non-identity commutator $s$.

Now as $N$ is an SLC-group we know that $\ast$ is a canonical involution on $N$. To prove the last statement we consider the following cases.
\begin{enumerate}[$i.$]
 \item If $g\in N\cap \zeta(G)$, we have that $g\in \zeta(N)$ and then 
$g^{\ast}=g$. 
\item If $g\in (G\setminus N)\setminus \zeta(G)$, then $g=ng_0$ for some $n\in N$. As $g_0$ is central in $G$, we have that $n\not\in \zeta(N)$.
Thus $g^{\ast}=(ng_0)^{\ast}=n^{\ast}g_0^{\ast}=snsg_0=ng_0=g$.
\item If $g\in (G\setminus N)\cap \zeta(G)$, then $g=ng_0$ for some $n\in N$. As $g\in \zeta(G)$, then $n\in \zeta(N)$. Thus
$$g^{\ast}=n^{\ast}g_0^{\ast}=nsg_0=sg.$$
\item Assume that $g\in N\setminus \zeta(G)$. It is easy to see that $g\not\in \zeta(N)$ and therefore $g^{\ast}=sg$.
\end{enumerate}

\end{proof}

We are now able to obtain necessary and sufficient conditions on $G$ and $N$ under which the group algebra $\mathbb{F}G$ is normal with respect
to the involution $\circledast$ given by expression \eqref{eq0}.

%

\begin{teor}
Let $\mathbb{F}$ be a field with $\car(\mathbb{F})\neq 2$ and let $G$ be a non-abelian group with involution $\ast$ and a non-identity orientation $\sigma$. Then, $\mathbb{F}G$ is normal with respect to the oriented group involution $\circledast$ if and only if one
of the following conditions holds:
\begin{enumerate}
 \item[$1.$] $N=ker(\sigma)$ is abelian, $\left(G\setminus N\right)\subset G^+$ and $n^\ast=a^{-1}na=ana^{-1}$, for all $n\in N$ and for all
      $a\in G\setminus N$;

 \item[$2.$] $N=ker(\sigma)$ and $G$ are LC-groups and there exists a unique non-identity commutator $s$ and $g_0\in (G\setminus N)\cap \zeta(G)$ such that 
      $g_0^\ast=sg_0$ and the involution $\ast$ is given by
      \begin{equation*}
g^{\ast}=
\begin{cases}
 g, & \text{ if $g\in N\cap \zeta(G)$ or $g\in (G\setminus N)\setminus \zeta(G)$;}\\
 sg, & \text{otherwise.}
\end{cases}
\end{equation*}

\end{enumerate}
\end{teor}

\begin{proof}
From lemmas $\ref{lem9}$, $\ref{lem10}$ and $\ref{lem11}$ and Proposition $\ref{prop1}$ the conditions are necessary. Now we prove
that they are also sufficient.
\begin{enumerate}
 \item[$1.$] Suppose $N=ker(\sigma)$ is abelian, $x^\ast=x$ for $x\in G\setminus N$ and $n^\ast=a^{-1}na=ana^{-1}$, for all $n\in N$ and for
      all $a\in G\setminus N$. We first note that $n^\ast$ is well-defined because if $a$ and $b$ are any two elements of $G\setminus N$,
      then $b=ma$ for some $m\in N$, so $b^{-1}nb=a^{-1}m^{-1}nma=a^{-1}na$, because $N$ is abelian.

      Now fix $a\notin N$. To show that $*$ is an involution, there are three cases to consider.

  \begin{enumerate}
   \item[$i.$] If $m, n\in N$, then because $N$ is abelian, we have that
        $$
        (mn)^\ast=a^{-1}mna=a^{-1}nma=(a^{-1}na)(a^{-1}ma)=n^\ast m^\ast.
        $$

   \item[$ii.$] If $m\in N$ and $x\notin N$, then $x=na$ for some $n\in N$. Since $mx\notin N$, from our assumption, $(mx)^\ast=mx$. Moreover,
        since $N$ is abelian, we obtain that
        $$
        x^*m^\ast=xa^{-1}ma=naa^{-1}ma=nma=mna=mx.
        $$
        Hence $(mx)^\ast=x^\ast m^\ast$.

   \item[$iii.$] If $x, y\notin N$, then $x=ma$ and $y=na$ for some $m, n\in N$. Since $xy\in N$, $(xy)^\ast=a^{-1}xya=a^{-1}(ma)(na)a=nama$
        (because $a^{-1}ma, na^2\in N$ and $N$ is abelian), while $y^\ast x^\ast=yx=nama$ also, i.e., $(xy)^\ast=y^\ast x^\ast$.

  \end{enumerate}

      Now, if $\alpha\in \mathbb{F}G$, then $\alpha=\beta +\gamma a$ with $\beta, \gamma\in \mathbb{F}N$ and $a\in G\setminus N$. So 
			
			$\alpha^{\circledast}=\beta^{\circledast}+(\gamma a)^{\circledast}$ , where $\beta^{\circledast}=a^{-1}\beta a$ and $(\gamma a)^{\circledast}=-\gamma a$. Then, as $\mathbb{F}N$ is a commutative group algebra, we obtain
			\begin{align*}
			\alpha^\circledast \alpha&=(a^{-1}\beta a-\gamma a)(\beta+\gamma a)\\
			&=a^{-1}\beta a \beta+ a^{-1}\beta a\gamma a-\gamma a\beta-\gamma a\gamma a\\
			&=a^{-1}\beta a \beta -\gamma a\gamma a,
			\end{align*}
			because $a^{-1}\beta a\gamma a=\gamma a\beta$. Similarly, 
			\begin{align*}
			\alpha\alpha^\circledast &=(\beta+\gamma a)(a^{-1}\beta a-\gamma a)\\
			&=\beta a^{-1}\beta a-\beta \gamma a+\gamma a a^{-1}\beta a-\gamma a\gamma a\\
			&=\beta a^{-1}\beta a -\gamma a\gamma a,
			\end{align*}
			because $\gamma a a^{-1}\beta a=\gamma \beta a=\beta \gamma a$. Therefore, as $a^{-1}\beta a \beta=\beta a^{-1}\beta a$, we conclude that $\mathbb{F}G$ is normal.

\item[$2.$] Let $\alpha=\beta + \gamma g_0\in \mathbb{F}G$ with $\beta,\gamma\in \mathbb{F}N$. Then 
\begin{align*}
\alpha&=\beta+\gamma g_0\\
&=\underbrace{\sum_{w\in \zeta(N)}\beta_w w}_{\beta_1}+\underbrace{\sum_{m\not\in \zeta(N)}\beta_m m}_{\beta_2}+\left(\underbrace{\sum_{z\in \zeta(N)}\gamma_z z}_{\gamma_1}+\underbrace{\sum_{n\not\in \zeta(N)}\gamma_n n}_{\gamma_2}\right)g_0\\
&=(\beta_1+\gamma_1g_0)+(\beta_2+\gamma_2g_0).
\end{align*}
Thus $\alpha^\circledast=(\beta_1-sg_0\gamma_1)+s(\beta_2-sg_0\gamma_2)$ and then 
\begin{align}\label{eq:teor4.1_eq1}
\begin{split}
\alpha\alpha^\circledast&=(\beta_1+\gamma_1g_0)(\beta_1-sg_0\gamma_1)+s(\beta_1+\gamma_1g_0)(\beta_2-sg_0\gamma_2)\\
&\qquad+(\beta_2+\gamma_2g_0)(\beta_1-sg_0\gamma_1)+s(\beta_2+\gamma_2g_0)(\beta_2-sg_0\gamma_2).
\end{split}
\end{align}
and

\begin{align}\label{eq:teor4.1_eq2}
\begin{split}
\alpha^\circledast\alpha&=(\beta_1-sg_0\gamma_1)(\beta_1+\gamma_1g_0)+(\beta_1-sg_0\gamma_1)(\beta_2+\gamma_2g_0)\\
&\qquad+s(\beta_2-sg_0\gamma_2)(\beta_1+\gamma_1g_0)+s(\beta_2-sg_0\gamma_2)(\beta_2+\gamma_2g_0).
\end{split}
\end{align}
Since $(\beta_1+\gamma_1g_0),(\beta_1-sg_0\gamma_1)\in \zeta(\mathbb{F}G)$, from equations \eqref{eq:teor4.1_eq1} and \eqref{eq:teor4.1_eq2}, we obtain that 
\begin{align*}
\alpha\alpha^\circledast-\alpha^\circledast\alpha&=s(\beta_2+\gamma_2g_0)(\beta_2-sg_0\gamma_2)-s(\beta_2-sg_0\gamma_2)(\beta_2+\gamma_2g_0)\\
&=s\beta_2^2-g_0\beta_2\gamma_2+sg_0\gamma_2\beta_2-g_0^2\gamma_2^2-s\beta_2^2-sg_0\beta_2\gamma_2+g_0\gamma_2\beta_2+g_0^2\gamma_2^2\\
&=-g_0(1+s)\beta_2\gamma_2+g_0(s+1)\gamma_2\beta_2\\
&=g_0(1+s)(\gamma_2\beta_2-\beta_2\gamma_2)\\
&=g_0(1+s)[\gamma_2,\beta_2].
\end{align*}

Furthermore, from the linearity of the commutator we have that
$$[\gamma_2,\beta_2]=\sum_{\substack{n,m\not\in\zeta(N) \\ nm\neq mn}}\lambda_{nm}[n,m]=\sum_{\substack{n,m\not\in\zeta(N) \\ nm\neq mn}}\lambda_{nm}nm(1-(m,n))$$
and as $(m,n)=s$ and $s^2=1$, therefore 
$$\alpha\alpha^\circledast-\alpha^\circledast\alpha=g_0(1+s)\sum_{\substack{n,m\in\zeta(N) \\ nm\neq mn}}\lambda_{nm}nm(1-s)=0.$$
Thus, $\mathbb{F}G$ is normal.

\end{enumerate}
\end{proof}

From the last result and \cite[Theorem 2.2]{OP} we obtain a similar result to Theorem \ref{teor3} in the setting of oriented group involutions.
\begin{teor}
Let $\mathbb{F}$ be a field with $\car(\mathbb{F})\neq 2$ and let $G$ a non-abelian  group with involution $\ast$ and non-identity orientation $\sigma$. Then  $\mathbb{F}G$ is normal if and only if $\mathbb{F}G^+$ is commutative.
\end{teor}
\section*{Acknowledgements}
The authors are grateful to their advisor Professor C\'esar Polcino Milies, who introduced them to the fascinating topic of group rings and of course for his timely suggestions, in particular they are indebted to him for his idea in
the construction of the special element $g_0$ in Lemma \ref{lem11}. Some results in this paper are part of
the first author's Ph.D. thesis, at Instituto de Matem\'atica e Estat\'istica of the
Universidade de S\~ao Paulo. This paper was written while the authors visited the Universidad Industrial de Santander and
Universidad de Nari\~no, and they thank the members of these institutions for their warm hospitality.
This research was partially supported by Vicerrectoria de Investigaciones, Postgrados y Relaciones Internacionales of the  Universidad de Nari\~no.

\bibliographystyle{plain}
\bibliography{bibliografia}  

\end{document}